\documentclass[a4paper,12pt]{article}

\usepackage{amssymb}
\usepackage{epsfig}
\usepackage{amsfonts}
\usepackage{amsmath}
\usepackage{euscript}
\usepackage{amscd}
\usepackage{amsthm}
\DeclareMathAlphabet{\mathpzc}{OT1}{pzc}{m}{it}

\newtheorem{thm}{Theorem}[section]
\newtheorem{lem}[thm]{Lemma}
\newtheorem{prop}[thm]{Proposition} 
\newtheorem{cor}[thm]{Corollary}

\newcommand{\bZ}{\mathbb Z}

\newcommand{\A}{\mathbb A}
\newcommand{\V}{\mathbb V}
\newcommand{\Ga}{\mathbb G}

\newcommand{\td}{\operatorname{tr.deg}}

\newcommand{\gr}{\operatorname{gr}}
\newcommand{\dk}{\operatorname{DK}}

\title{On Zariski's Cancellation Problem in positive characteristic}
\author{Neena Gupta\\
{\small{\it Stat-Math Unit, Indian Statistical Institute}}\\
{\small{\it 203 B.T. Road, Kolkata 700108, India}}\\
{\small{\it e-mail: neenag@isical.ac.in, rnanina@gmail.com}}}

\begin{document}

\date{}
\maketitle

\begin{abstract}
In this paper we shall show that when $k$ is a field of positive characteristic
the affine space $\A^n_k$ is not cancellative for {\it any} $n \ge 3$. 
\end{abstract}

\section{Introduction}
Let $k$ be an algebraically closed field.
The Zariski Cancellation Problem for Affine Spaces 
asks whether the affine space $\A^n_k$ is cancellative, i.e.,  
if $\V$ is an affine $k$-variety such that 
$\V \times \A^1_k \cong \A^{n+1}_k$,
does it follow that $\V \cong \A_k^n$? Equivalently, 
if $A$ is an affine $k$-algebra
such that $A[X]$ is isomorphic to the polynomial ring $k[X_1, \dots, X_{n+1}]$,
does it follow that $A$ is isomorphic to $k[X_1, \dots, X_n]$?

The affine line $\A^1_k$ was shown to be cancellative by 
S. S. Abhyankar, P. Eakin and W. J. Heinzer
(\cite{AEH}) and the affine plane $\A^2_k$ was shown to be cancellative by 
T. Fujita, M. Miyanishi and T. Sugie
(\cite{F}, \cite{MS}) in characteristic zero and 
by P. Russell (\cite{R}) in positive characteristic. 
However, in \cite{G}, the author showed when ch. $k>0$,
the affine space $\A^3_k$ is not cancellative by proving that
a threefold constructed by Asanuma in \cite{A} is  
not isomorphic to the polynomial ring $k[X_1, X_2, X_3]$.  

In this paper, we shall show when ch. $k>0$,
the affine space $\A^n_k$ is not cancellative  
for any $n\ge 3$ (Theorem \ref{zar}). 
This completely settles the Zariski's Cancellation problem 
in positive characteristic.

\section{Preliminaries}\label{pri}
We shall use the notation $R^{[n]}$ 
for a polynomial ring in $n$ variables over a ring $R$.

We shall also use the following term from affine algebraic geometry.
\smallskip

\noindent
{\bf Definition.} 
An element $f \in k[Z, T]$ is called a {\it line} if $k[Z, T]/(f) = k^{[1]}$.
A line $f$ is called  a {\it non-trivial line} if $k[Z, T]\neq k[f]^{[1]}$.

We recall the definition of an exponential map 
(a formulation of the concept of $\Ga_a$-action)
and associated invariants.

\smallskip

\noindent
{\bf Definition.}
Let $A$ be a $k$-algebra and let $\phi: A \to A^{[1]}$ be a 
$k$-algebra homomorphism. For an indeterminate $U$ over $A$, let the notation 
$\phi_U$ denote the map $\phi: A \to A[U]$.  
$\phi$ is said to be an {\it exponential map on $A$} if $\phi$ satisfies the following two properties:
\begin{enumerate}
 \item [\rm (i)] $\varepsilon_0 \phi_U$ is identity on $A$, where 
 $\varepsilon_0: A[U] \to A$ is the evaluation at $U = 0$.
 \item[\rm (ii)] $\phi_V \phi_U = \phi_{V+U}$, where 
 $\phi_V: A \to A[V]$ is extended to a homomorphism 
 $\phi_V: A[U] \to A[V,U]$ by  setting $\phi_V(U)= U$.
 \end{enumerate}
The ring of $\phi$-invariants of an exponential map $\phi$ on $A$
is a subring of $A$ given by 
$$
A^{\phi} = \{a \in A\,| \,\phi (a) = a\}.
$$
An exponential map $\phi$ is said to be {\it non-trivial} if $A^{\phi} \neq A$.
For an affine domain $A$ over a field $k$, let 
${\rm EXP} (A)$ denote the set of all exponential maps on $A$.
The {\it Derksen invariant} of $A$ is a subring of $A$ defined by
$$
{\rm DK} (A) = k[ f \,| \, f \in A^{\phi}, \phi ~\text{a non-trivial exponential map}].
$$ 


We recall below a crucial observation (cf. \cite[Lemma 2.4]{G}).

\begin{lem}\label{r1}
Let $k$ be a field and $A= k^{[n]}$, where $ n > 1$. Then ${\rm DK} (A) = A$.
\end{lem}





We shall also use the following result proved in \cite[Lemma 3.3]{G}.

\begin{lem}\label{lem1}
Let $B$ be an affine domain over an infinite field $k$. 
Let $f \in B$ be such that  $f - \lambda $ is a prime element of $B$ 
for infinitely many $\lambda \in k$.  Let $\phi$ be
a non-trivial exponential map on $B$ such that
$f \in B^{\phi}$. Then there exist infinitely many $\beta \in k$ such that each $f - \beta$ 
is a prime element of $B$ and $\phi$
induces a non-trivial exponential map  $\hat{\phi}$ on $B/(f - \beta)$. 
Moreover, $B^{\phi}/(f - \beta)B^{\phi}$ is contained in $(B/(f - \beta))^{\hat{\phi}}$. 
\end{lem}

We shall also use the following result proved in \cite[Theorem 3.7]{G2}.

\begin{thm}\label{dka}
Let $k$ be a field and $A$ be an integral domain defined by 
$$
A= k[X,Y,Z,T]/(X^m Y - F(X, Z, T)), {\text{~~where~~}} m > 1.
$$
Set $f(Z, T):= F(0, Z, T)$. Let $x$, $y$, $z$ and $t$ denote, 
respectively, the images of $X$, $Y$, $Z$ and $T$ in $A$. 
Suppose that $\dk(A) \neq k[x,z,t]$. 
Then the following statements hold.
\begin{enumerate}
\item [\rm (i)] There exist $Z_1, T_1 \in k[Z, T]$ and $a_0, a_1 \in k^{[1]}$ such that
$k[Z, T]=k[Z_1, T_1]$ and $f(Z, T) = a_0(Z_1) + a_1(Z_1)T_1$. 
\item [\rm(ii)] If $k[Z, T]/(f)= k^{[1]}$, then $k[Z, T]=k[f]^{[1]}$.
\end{enumerate}
\end{thm}

An exponential map $\phi$ on a graded ring $A$ is said to be {\it homogeneous}
if $\phi: A \to A[U]$ becomes homogeneous 
when $A[U]$ is given a grading induced from $A$ such that $U$ is a homogeneous element. 


We state below a result on homogenization of exponential maps due to 
H. Derksen, O. Hadas and L. Makar-Limanov (\cite{DHM}, cf. \cite[Theorem 2.3]{G}).

\begin{thm}\label{MDH}
Let $A$ be an affine domain over a field $k$ with an admissible 
proper $\bZ$-filtration and $\gr(A)$ the induced $\bZ$-graded domain.
Let $\phi$ be a non-trivial exponential map on $A$.
Then $\phi$ induces a non-trivial homogeneous exponential map $\bar{\phi}$ on 
$\gr (A)$ such that $\rho( {A^{\phi}}) \subseteq {\gr (A)}^{\bar{\phi}}$.
\end{thm}

\section{Main Theorems}

Throughout the paper, $k$ will denote a field 
(of any characteristic unless otherwise specified) and 
$A$ an integral domain defined by 
$$
A= k[X_1, \dots, X_{m},Y,Z,T]/({X_1}^{r_1} \cdots {X_m}^{r_m} Y - F(X_1, \dots, X_m, Z, T)),
$$
where $r_i >1$ for each $i$, $1\le i\le m$. 
The images of $X_1$, $\dots$, $X_m$, $Y$, $Z$ and $T$ in $A$ will be denoted by
$x_1$, $\dots$, $x_m$, $y$, $z$ and $t$ respectively. 

Set $B: = k[x_1,\dots, x_m, z, t] (= k^{[m+2]})$.
We note that $B \hookrightarrow A \hookrightarrow B[({x_1\cdots x_m})^{-1}]$.
For each $m$-tuple $(q_1, \dots, q_m) \in \bZ^m$, 
consider the $\bZ$-grading on $B[({x_1\cdots x_m})^{-1}]$
given by 
$$
B[({x_1\cdots x_m})^{-1}]= \bigoplus_{i \in \bZ}B_i, \text{~where~}
B_i = \bigoplus_{(i_1, \dots, i_m) \in \bZ^m, q_1i_1+q_2i_2+ \dots+ q_mi_m = i}k[z, t]{x_1}^{i_1}{x_2}^{i_2}\dots{x_m}^{i_m}.
$$
Each element $a \in B[({x_1\cdots x_m})^{-1}]$ can be uniquely written as 
$a= \sum_{{\ell_a} \le j\le {u_a}} a_j$, where $a_j \in B_j$. 
(Note that if $a \in B$ then $a_j \in B$ for each $j$, ${\ell_a} \le j\le {u_a}$.)
We call $u_a$ the degree of $a$ and $a_{u_a}$ the leading homogeneous summand of $a$.
 
Using this grading on $B[({x_1\cdots x_m})^{-1}]$, we can define a proper
$\bZ$-filtration $\{A_n\}_{n \in \bZ}$ on $A$
by setting $A_n := A \cap \bigoplus_{i \le n}B_i$. 
Then $x_j \in A_{q_j}\setminus A_{q_j-1}$, $1\le j\le m$ 
and $z, t \in A_{0}\setminus A_{-1}$. 
Since $A$ is an integral domain, we have $F \neq 0$ and hence $F_{u_F} \neq 0$. 
Thus, $y \in A_b\setminus A_{b-1}$, where 
$b= u_F-(q_1r_1+ \cdots + q_mr_m)$. 
Let $\gr{A}$ denote the induced graded ring 
$\bigoplus_{n \in \bZ}A_{n}/ A_{n-1}$. 

Note that each element $h \in A$ can be uniquely written as sum of monomials of the form 
\begin{equation}\label{eqn}
{x_1}^{i_1} \cdots {x_m}^{i_m} z^{j_1}t^{j_2}y^{\ell}, 
\end{equation}
where $i_j \ge 0$, $1\le j \le m$ and $\ell \ge 0$ satisfying  
that if $\ell > 0$ then $i_{s} < r_s$ for at least one $s$, $1\le s \le m$. 
Therefore, it can be seen that, if 
$h \in A_{n}\setminus A_{n-1}$ then $h$ can be uniquely written as 
sum of monomials of the form as in equation (\ref{eqn})
and each of these monomials lies  in $A_n$.   
Thus, the filtration defined on $A$ is admissible with the
generating set $\Gamma := \{x_1, \dots, x_m, y, z, t\}$
and so $\gr{A}$ is generated by image of $\Gamma$ in $\gr{A}$
(cf. \cite[Remark 2.2 (2)]{G}).

We now exhibit a structure of $\gr{A}$ when $F_{u_F}$ is not divisible by $x_j$ for any $j$. 

\begin{lem}\label{fil}
Suppose that $F_{u_F}$ is not divisible by $x_j$ for any $j$, $1\le j\le m$, 
then $\gr{A}$ is isomorphic to 
$$
D = \dfrac{k[X_1, \dots, X_m,Y,Z,T]}{({X_1}^{r_1} \cdots {X_m}^{r_m} Y - F(X_1, \dots, X_m, Z, T)_{u_F})}.
$$ 
\end{lem}
\begin{proof}
For $a \in A$, let $\gr(a)$ denote the image of $a$ in $\gr{A}$.
Then, as discussed above, $\gr{A}$ is generated by $\gr(x_1)$, $\dots$, 
$\gr(x_m)$,  $\gr(y)$, $\gr(z)$ and $\gr(t)$.
Note that if $a \in B (\subseteq A)$, then $\gr(a) = \gr(a_{u_a})$. 
As ${x_1}^{r_1} \cdots {x_m}^{r_m}{y} (= F)\in A_{u_F} \setminus A_{u_F-1}$   
and hence ${x_1}^{r_1} \cdots {x_m}^{r_m}{y}- F(x_1, \dots, x_m, z, t)_{u_F} \in A_{u_F-1}$. Therefore, 
$$
\gr({x_1})^{r_1} \cdots \gr({x_m})^{r_m}\gr({y})- \gr(F_{u_F}) =0  {\text{~~in~~}} \gr{A}.$$
Since $F_{u_F}$ is not divisible by $x_j$ for any $j$, 
$1\le j\le m$, $D$ is an integral domain. 
As $\gr{A}$ can be identified with a subring of 
$\gr (B[({x_1\cdots x_m})^{-1}])  \cong B[({x_1\cdots x_m})^{-1}]$, 
we see that the elements $\gr(x_1),\dots, \gr(x_m),\gr(z),\gr(t)$ of $\gr{A}$
are algebraically independent over $k$. Hence $\gr{A} \cong D$. 
\end{proof}

\begin{lem}\label{dkA}
We have $B (=k[x_1, \dots, x_m, z, t])\subseteq \dk(A)$.
\end{lem}
\begin{proof}
Define $\phi_1$ by $\phi_1(x_j) =x_j$ for each $j$, $1\le j\le m$, 
$\phi_1(z) = z$, $\phi_1(t) =t+{x_1}^{r_1} \cdots {x_m}^{r_m}U$ and 
\begin{equation*}
\phi_1(y) = \frac{F(x_1, \dots, x_m,z,t+{x_1}^{r_1} \cdots 
{x_m}^{r_m}U)}{{x_1}^{r_1} \cdots {x_m}^{r_m}} = y + U\alpha(x_1, \dots, x_m, z, t, U)
\end{equation*}
and define $\phi_2$ by $\phi_2(x_j) =x_j$ for each $j$, $1\le j\le m$, 
$\phi_2(t) = t$, $\phi_2(z) =z+{x_1}^{r_1} \cdots {x_m}^{r_m}U$,
\begin{equation*}
\phi_2(y) = \frac{F(x_1, \dots, x_m, z+{x_1}^{r_1} \cdots {x_m}^{r_m}U, t)}
{{x_1}^{r_1} \cdots {x_m}^{r_m}} = y+ U \beta(x_1, \dots, x_m, z, t, U). 
\end{equation*}
Note that $\alpha(x_1, \dots, x_m, z, t, U), \beta(x_1, \dots, x_m, z, t, U) \in k[x_1, \dots, x_m, z, t, U]$
and that $k[x_1, \dots, x_m,z]$  and $k[x_1, \dots, x_m, t]$ are algebraically closed in $A$ 
of transcendence degree $m+1$ over $k$.
It then follows that $\phi_1$ and $\phi_2$ are nontrivial exponential maps on $A$ 
with $A^{\phi_1} = k[x_1, \dots, x_m,z]$ and $A^{\phi_2} = k[x_1, \dots, x_m,t]$.
Hence $k[x_1, \dots, x_m, z, t]\subseteq \dk(A)$.
\end{proof}

We now prove a generalisation of Theorem \ref{dka}.

\begin{prop}\label{main}
Suppose that $f(Z, T):= F(0,\dots, 0, Z, T) \neq 0$ and that $\dk(A) = A$.
Then the following statements hold.
\begin{enumerate}
\item [\rm (i)] There exist $Z_1, T_1 \in k[Z, T]$ and $a_0, a_1 \in k^{[1]}$ such that
$k[Z, T]=k[Z_1, T_1]$ and $f(Z, T) = a_0(Z_1) + a_1(Z_1)T_1$. 
\item [\rm (ii)] Suppose that 
$k[Z, T]/(f) = k^{[1]}$. Then $k[Z, T] = k[f]^{[1]}$. 
\end{enumerate}
\end{prop}
\begin{proof}
(i) We prove the result by induction on $m$. The result is true for $m= 1$ by Theorem \ref{dka}. 
Suppose that $m >1$. 
Set $B:= k[x_1, x_2, \dots, x_m, z, t] (\subseteq A)$.  Since $\dk(A) = A$, 
there exists an exponential map $\phi$ on $A$ such that $A^{\phi} \nsubseteq B$. 
Let $g \in A^{\phi} \setminus B$.  Since $g \notin B$,
by equations (\ref{eqn}),
there exists a monomial in $g$ which is of the form 
${x_1}^{i_1} \dots {x_m}^{i_m} z^{j_1}t^{j_2}y^{\ell}$
where $\ell >0$ and $i_{s}< r_s$ for some $s$, $1\le s \le m$.
Without loss of generality, we may assume that $s=1$. 

Consider the proper $\bZ$-filtration on $A$
with respect to $(-1, 0, \dots, 0) \in \bZ^m$
and let $\overline{A}$ denote the induced graded ring. 
For $h \in A$, let $\bar{h}$ denote the image of $h$ in $\overline{A}$.
Since $A$ is an integral domain, we have
$F(0, x_2, \dots, x_m, z, t) \neq 0$ and so, 
$F(x_1, \dots, x_m, z, t)_{u_F}= 
F(0, x_2, \dots, {x_m}, {z}, {t})= G$ (say). 
Since $f(Z, T) \neq 0$,  by Lemma \ref{fil}, we have 
$$
\overline{A} \cong 
k[X_1, \dots, X_{m},Y,Z,T]/({X_1}^{r_1} \cdots {X_m}^{r_m} Y - G).
$$
By Theorem \ref{MDH}, $\phi$ induces a non-trivial homogeneous
exponential map $\bar{\phi}$ on $\overline{A}$ such that $\bar{g} \in \overline{A}^{\bar{\phi}}$.
By the choice of $g$ and the filtration defined on $A$, 
we have $\bar{y} \mid \bar{g}$. 
Since $\overline{A}^{\bar{\phi}}$ is factorially closed in $\overline{A}$ 
(cf. \cite[Lemma 2.1 (i)]{G}),
it follows that $\bar{y} \in \overline{A}^{\bar{\phi}}$.

We now consider the proper $\bZ$-filtration on $\overline{A}$ 
with respect to $(-1, -1, \dots, -1) \in \bZ^m$
and let $\widetilde{A}$ denote the induced graded ring. 
For $\bar{h} \in \overline{A}$, let $\widetilde{h}$ denote 
the image of $\bar{h}$ in $\widetilde{A}$.
Since $f(Z, T) \neq 0$, we have 
${G}_{u_G}= f(z, t)$
and hence by Lemma \ref{fil},  
$$
\widetilde{A} \cong k[X_1, \dots, X_{m},Y,Z,T]/
({X_1}^{r_1} \cdots {X_m}^{r_m} Y - f(Z, T)). 
$$
Again by Theorem \ref{MDH}, $\bar{\phi}$ induces a non-trivial homogeneous
exponential map $\widetilde{\phi}$ on $\widetilde{A}$ such that $\widetilde{y} \in \widetilde{A}^{\widetilde{\phi}}$.
Since $\td_k \widetilde{A}^{\widetilde{\phi}} = m+1$, there exist $m$ algebraically independent elements
in $\widetilde{A}^{\widetilde{\phi}}$ over $k[\widetilde{y}]$.  

If $\widetilde{A}^{\widetilde{\phi}} \subseteq k[\widetilde{y},\widetilde{z},\widetilde{t}]$, 
then since $m \ge 2$ and $\widetilde{A}^{\widetilde{\phi}}$ is algebraically closed in $\widetilde{A}$, 
we have $\widetilde{A}^{\widetilde{\phi}} = k[\widetilde{y},\widetilde{z},\widetilde{t}]$. 
Since $\widetilde{x_1}^{r_1} \cdots \widetilde{x_m}^{r_m} \widetilde{y}  (= f(\widetilde{z}, \widetilde{t})) \in  
k[\widetilde{y},\widetilde{z},\widetilde{t}]$, 
we have $\widetilde{x_1}, \dots, \widetilde{x_m}, \widetilde{y} \in \widetilde{A}^{\widetilde{\phi}}$ as 
$\widetilde{A}^{\widetilde{\phi}}$ is factorially closed in $\widetilde{A}$. This would contradict
the fact that $\widetilde{\phi}$ is non-trivial. Thus, there exists an homogeneous element 
$h \in \widetilde{A}^{\widetilde{\phi}} \setminus k[\widetilde{y},\widetilde{z},\widetilde{t}]$. 
Hence,  $h$ contains a monomial which is divisible by $\widetilde{x_i}$ for some $i$, $1\le i\le m$. Without loss of generality
we may assume that $\widetilde{x_2 }$ divides a monomial of $h$.  

Now again consider the proper $\bZ$-filtration on $\widetilde{A}$ 
with respect to the $m$-tuple $(0, 1, 0, \dots, 0) \in \bZ^{m}$
defined as in Lemma \ref{fil} and the induced graded ring $\gr{\widetilde{A}}$ of $\widetilde{A}$.  
Then $\gr{\widetilde{A}} \cong \widetilde{A}$. 
For $\widetilde{a} \in\widetilde{A}$, 
let $\gr(a)$ denote the image of $\widetilde{a}$ in $\gr{\widetilde{A}}$. 
Then $\gr(x_2) \mid \gr(h)$ in $\gr{A}$. Again by Theorem \ref{MDH}, 
$\widetilde{\phi}$ induces a non-trivial homogeneous
exponential map $\gr{\phi}$ on $\gr{\widetilde{A}}$ such that 
$\gr(y)$ and $\gr(h) \in {\gr{A}}^{\gr{\phi}}$.
Since $\gr(x_2) \mid \gr(h)$ in $\gr{A}$, we have $\gr(x_2) \in {\gr{A}}^{\gr{\phi}}$.  

Let $F$ be an algebraic closure of the field $k$. 
Then $\gr{\phi}$ induces a non-trivial exponential map $\psi$ on 
$$
E=\gr{\widetilde{A}} \otimes_k F\cong \dfrac{F[X_1,\dots, X_m, Y, Z, T]}
{({X_1}^{r_1} \cdots {X_m}^{r_m} Y - f(Z, T))} = F[\hat{x_1}, \dots, \hat{x_m}, \hat{z}, \hat{t}, \hat{y}],
$$
such that $F[\hat{y}, \hat{x_2}]\subseteq E^{{\psi}}$. 
Since $\hat{x_2} -\lambda$ is a prime element in $E$ 
for every $\lambda \in F^*$, by Lemma \ref{lem1}, 
$\psi$ induces a non-trivial exponential map on  
$$
E/ (\hat{x_2} -\lambda)E (\cong \dfrac{F[X_1, X_3, \dots, X_{m},Y,Z,T]}
{(\lambda^{r_2}{X_1}^{r_1}{X_{3}}^{r_3}\cdots{X_m}^{r_m} Y -f(Z, T))}
$$
for some $\lambda \in F^*$ 
such that the image of $\hat{y}$ lies in the ring of invariants. 
Therefore, using Lemma \ref{dkA}, 
we have $\dk(E/ (\hat{x_2} -\lambda)E)= E/ (\hat{x_2} -\lambda)E$. 
We are thus through by induction on $m$. 

(ii) Suppose that 
$f(Z, T)$ is a line in $k[Z, T]$. 
Then $A/(x_1, \dots, x_m) = k[Y, Z, T]/(f(Z, T)) \cong k^{[2]}$ 
and hence $(A/(x_1, \dots, x_m))^*= k^*$. By (i) above,
there exist $Z_1, T_1 \in k[Z, T]$ and $a_0, a_1 \in k^{[1]}$ such that
$k[Z, T]=k[Z_1, T_1]$ and $f(Z, T) = a_0(Z_1) + a_1(Z_1)T_1$. 
If $a_1(Z_1)=0$, then $f(Z, T)= a_0(Z_1)$ must be a linear polynomial in $Z_1$ 
and hence a variable in $k[Z, T]$. 
Now suppose $a_1(Z_1) \neq 0$. As $f(Z, T)$ is irreducible in $k[Z, T]$,
$a_0(Z_1)$ and $a_1(Z_1)$ are coprime in $k[Z_1]$.
Hence $A/(x_1, \dots, x_m) \cong k[Z_1, \frac{1}{a_1(Z_1)}]^{[1]}$ 
and since $(A/(x_1, \dots, x_m))^*= k^*$, we have  $a_1(Z_1) \in k^*$. 
This again implies that $f(Z, T)$ is a variable in $k[Z, T]$.
\end{proof}

We now deduce a result for a field of positive characteristic. 

\begin{cor}\label{ce}
Suppose that ch. $k >0$ and that 
$f(Z,T) \in k[Z, T]$ is a non-trivial line in $k[Z, T]$. 
Then, $A \ncong k^{[m+2]}$. 
\end{cor}
\begin{proof}
Suppose, if possible that $A \cong k^{[m+2]}$. Then, by Lemma \ref{r1}, $\dk(A)= A$. Therefore,
by Proposition \ref{main}(ii), $k[Z, T]= k[f]^{[1]}$. This contradicts the given hypothesis. 
\end{proof}

\begin{thm}\label{sta}
Let $k$ be a field of any characteristic and
$A$ an integral domain defined by 
$$
A= k[X_1, \dots, X_{m},Y,Z,T]/({X_1}^{r_1} \cdots {X_m}^{r_m} Y - f(Z, T)),
$$
where $r_i >1$ for each $i$. 
Suppose that $f(Z, T)$ is a line in $k[Z, T]$. Then 
$$
A^{[1]}=k[X_1, \dots, X_m]^{[3]} \cong k^{[m+3]}.
$$
\end{thm}
\begin{proof}
Let $h \in k[Z, T]$ be such that $k[Z, T]= k[h] +(f)k[Z, T]$.
Let $P , Q \in k^{[1]}$ be such that
$Z= P(h) + fP_1(Z, T)$ and 
$T = Q(h) + fQ_1(Z, T)$ for some $P_1, Q_1 \in k[Z, T]$. 
Let $W$ be an indeterminate over $A$.
Set 
\begin{eqnarray*}
W_1 : &=& {X_1}^{r_1} \cdots {X_m}^{r_m} W + h(Z, T) \\
Z_1:&=& (Z- P(W_1))/{X_1}^{r_1} \cdots {X_m}^{r_m}\\
T_1: &=& (T- Q(W_1))/{X_1}^{r_1} \cdots {X_m}^{r_m}. 
\end{eqnarray*}
Let $y$ be the image of $Y$ in $A$. 
We show that $A[W]= k[X_1, \dots, X_m, Z_1,T_1, W_1]$.
Set $B: = k[X_1,\dots, X_m, Z_1,T_1, W_1]$. We have 
\begin{eqnarray*}
 Z &=& P(W_1)+ {X_1}^{r_1} \cdots {X_m}^{r_m}Z_1, \\
 T &=& Q(W_1)+ {X_1}^{r_1} \cdots {X_m}^{r_m}T_1,\\
 y &=& \frac{f(Z,T)}{{X_1}^{r_1} \cdots {X_m}^{r_m}} = 
\frac{f({X_1}^{r_1} \cdots {X_m}^{r_m}Z_1 + P(W_1),{X_1}^{r_1} \cdots {X_m}^{r_m} T_1 + Q(W_1))}{{X_1}^{r_1} \cdots {X_m}^{r_m}} \\
   &=& \frac{f(P(W_1), Q(W_1))}{{X_1}^{r_1} \cdots {X_m}^{r_m}} + \alpha(X_1,\dots, X_m, Z_1, T_1, W_1),\\
 W &=& \frac{W_1- h(Z, T)}{{X_1}^{r_1} \cdots {X_m}^{r_m}} = \frac{W_1- h({X_1}^{r_1} \cdots {X_m}^{r_m}Z_1 + P(W_1), 
{X_1}^{r_1} \cdots {X_m}^{r_m}T_1 + Q(W_1))}{{X_1}^{r_1} \cdots {X_m}^{r_m}} \\
   &=& \frac{W_1- h(P(W_1), Q(W_1))}{{X_1}^{r_1} \cdots {X_m}^{r_m}} + \beta(X_1, \dots, X_m, Z_1, T_1, W_1)
\end{eqnarray*}
for some $\alpha , \beta \in B$. 
Since $f(P(W_1), Q(W_1)) = 0$ and $h(P(W_1), Q(W_1)) = W_1$, we see that $y, W \in B$.
Hence, $A[W] \subseteq B$.  
We now show that $B \subseteq A[W]$. We have,
\begin{eqnarray*}
Z_1 &= & \frac{Z- P(W_1)}{{X_1}^{r_1} \cdots {X_m}^{r_m}} 
= \frac{Z- P({X_1}^{r_1} \cdots {X_m}^{r_m} W + h(Z, T))}{{X_1}^{r_1} \cdots {X_m}^{r_m}} \\
&=& \frac{Z-P(h(Z, T))}{{X_1}^{r_1} \cdots {X_m}^{r_m}} + \gamma(X_1, \dots, X_m, Z, T, W)=
\frac{f(Z, T)P_1(Z, T)}{{X_1}^{r_1} \cdots {X_m}^{r_m}}+\gamma(X_1,\dots, X_m, Z, T, W)\\
&=& P_1(Z, T) y + \gamma(X_1,\dots, X_m, Z, T, W),
\end{eqnarray*}
and
\begin{eqnarray*}
T_1 &= & \frac{T- Q(W_1)}{{X_1}^{r_1} \cdots {X_m}^{r_m}} = 
\frac{T- Q({X_1}^{r_1} \cdots {X_m}^{r_m}W + h(Z, T))}{{X_1}^{r_1} \cdots {X_m}^{r_m}} \\
&=& \frac{T-Q(h(Z, T))}{{X_1}^{r_1} \cdots {X_m}^{r_m}} + \delta(X_1,\dots, X_m, Z, T, W)= 
\frac{f(Z, T)Q_1(Z, T)}{{X_1}^{r_1} \cdots {X_m}^{r_m}}+\delta(X_1,\dots, X_m, Z, T, W)\\
&=& Q_1(Z, T) y + \delta(X_1,\dots, X_m, Z, T, W)
\end{eqnarray*}
for some $\gamma, \delta \in A[W]$. 
Thus, $Z_1, T_1 \in A[W]$. Hence $B \subseteq A[W]$. 
Since $B = k[X_1, \dots, X_m]^{[3]}$,  the result follows.
\end{proof}

\begin{thm}\label{zar}
 When $k$ is a field of positive characteristic, 
Zariski's Cancellation Conjecture does not hold for the affine $n$-space
$\A^n_k$ for any $n \ge 3$. 
\end{thm}
\begin{proof}
There exists non-trivial line in $k[Z, T]$ when ch $k$ $= p>0$. For instance, we may take the 
Nagata's line $f(Z, T)= Z^{p^2} + T + T^{qp}$, where $q$ is a prime other than $p$.
The result now follows from Theorem \ref{sta} and Corollary \ref{ce}.  
\end{proof}

\small{

}
\end{document}